\documentclass[11pt, a4paper]{article}

\usepackage[margin=2.5cm]{geometry}\usepackage{graphicx,tikz,caption,subcaption}
\usepackage{fullpage}
\usepackage{hyperref}
\usepackage{enumitem,verbatim}
\usepackage{hyperref}
\usepackage{amsmath}
\usepackage{amssymb}
\usepackage{amsfonts}
\usepackage{amsthm}
\usepackage[capitalise]{cleveref}

\newtheorem{theorem}{Theorem}[section]
\newtheorem{lemma}[theorem]{Lemma}
\newtheorem{definition}[theorem]{Definition}
\newtheorem*{defn*}{Definition}
\newtheorem{claim}[theorem]{Claim}

\newtheorem{observation}[theorem]{Observation}
\newtheorem*{observation*}{Observation}

\newtheorem{question}[theorem]{Question}
\newtheorem*{question*}{Question}
\newtheorem{innerdefinition}{Definition}
\newenvironment{definition*}
  {
   \innerdefinition}
  {\endinnerdefinition}

\newcommand{\ex}{\mathrm{ex}}

\usepackage{tikz}
\tikzstyle{aNode} = [circle, fill = black]
\tikzstyle{bNode} = [circle,draw = black, thick]

\newcommand{\ppoints}[1]{%
\begin{tikzpicture}[inner sep = 0.7pt, #1]%
\node (1) at (0,-2) [aNode]{};
\node (3) at (1.5,-2) [aNode]{};
\node (2) at (0.75,-1) [aNode]{};
\end{tikzpicture}%
}

\def\points{\ppoints{scale=0.08}}
\def\co{\mathrm{co}}

\usetikzlibrary{patterns}
\usetikzlibrary{shapes}
\usetikzlibrary{decorations.pathreplacing}
\usetikzlibrary{decorations.pathmorphing}
\usetikzlibrary{positioning}

\title{Vanishing codegree Tur\'{a}n density implies vanishing uniform Tur\'{a}n density}
\author{Laihao Ding\thanks{School of Mathematics and Statistics, and Key Laboratory of Nonlinear Analysis \& Applications (Ministry of Education), Central China Normal University, Wuhan 430079, China and Extremal Combinatorics and Probability Group (ECOPRO),  Institute for Basic Science (IBS), Daejeon, South Korea. Supported by the National Nature Science Foundation of China (11901226), the China Scholarship Council and the Institute for Basic Science (IBS-R029-C4). \textbf{Email address}: dinglaihao@ccnu.edu.cn}
\and Hong Liu\thanks{Extremal Combinatorics and Probability Group (ECOPRO),  Institute for Basic Science (IBS), Daejeon, South Korea. Supported by the Institute for Basic Science (IBS-R029-C4). \textbf{Email address}: hongliu@ibs.re.kr}
\and Shuaichao Wang\thanks{Center for Combinatorics and LPMC, Nankai University, Tianjin 300071, China and Extremal Combinatorics and Probability Group (ECOPRO),  Institute for Basic Science (IBS), Daejeon, South Korea. Supported by the China Scholarship Council and Institute for Basic Science (IBS-R029-C4). \textbf{Email address}: wsc17746316863@163.com } 
\and Haotian Yang\thanks{Taishan College, Shandong University, Jinan 250100, China and Extremal Combinatorics and Probability Group (ECOPRO),  Institute for Basic Science (IBS), Daejeon, South Korea. Supported by Seed Fund Program for International Research Cooperation of Shandong University and Institute for Basic Science (IBS-R029-C4). \textbf{Email address}: 202017091012@mail.sdu.edu.cn }}

\begin{document}
\maketitle

\begin{abstract}
 For a $k$-uniform hypergraph (or simply $k$-graph) $F$, the codegree Tur\'{a}n density $\pi_{\co}(F)$ is the infimum over all $\alpha$ such that any $n$-vertex $k$-graph $H$ with every $(k-1)$-subset of $V(H)$ contained in at least $\alpha n$ edges has a copy of $F$. The uniform Tur\'{a}n density $\pi_{\points}(F)$ is the supremum over all $d$ such that there are infinitely many $F$-free $k$-graphs $H$ satisfying that any linear-size subhypergraph of $H$ has edge density at least $d$. Falgas-Ravry, Pikhurko, Vaughan and Volec [J. London Math. Soc., 2023] asked whether for every $3$-graph $F$, $\pi_{\points}(F)\leq\pi_{\mathrm{co}}(F)$. We provide a positive answer to this question provided that $\pi_{\mathrm{co}}(F)=0$. Our proof relies on a random geometric construction and a new formulation of the characterization of $3$-graphs with vanishing uniform Tur\'{a}n density due to Reiher, R{\"o}dl and Schacht [J. London Math. Soc., 2018]. Along the way, we answer a question of Falgas-Ravry, Pikhurko, Vaughan and Volec about subhypergraphs with linear minimum codegree in uniformly dense hypergraphs in the negative.
\end{abstract}
\bigskip

\section{Introduction}
Given a $k$-uniform hypergraph $F$ (or simply \emph{$k$-graph}), one of the central problems in extremal combinatorics is to determine the maximum number of edges in an $n$-vertex $k$-graph $H$ containing no copy of $F$, which is called the \emph{Tur\'{a}n number} of $F$ and is denoted by ${\rm ex}(n,F)$. Targeting on the limit behavior, one may define the \emph{Tur\'{a}n density} 
$$\pi(F):=\lim\limits_{n\to \infty}\frac{\ex(n,F)}{\binom{n}{k}}.$$
An averaging argument shows that this limit always exists.
While Tur\'{a}n densities are well-understood for graphs (i.e., $2$-graphs), determining the Tur\'{a}n density of a $k$-graph becomes notoriously difficult when $k\geq 3$. Despite much efforts and attempts, even the Tur\'{a}n densities of the $3$-graphs on four vertices with three and four edges, denoted by $K_4^{(3)-}$ and $K_4^{(3)}$ respectively, are still unknown.

Erd\H{o}s and S\'{o}s \cite{erdossos} considered a variation requiring the host hypergraph to be uniformly dense, formally defined as follows. For real numbers $d\in[0,1]$ and $\eta>0$, an $n$-vertex $k$-graph $H$ is said to be \emph{uniformly $(d,\eta)$-dense} if for all $U\subseteq V(H)$, it holds that $\left| \binom{U}{k}\cap E(H)\right|\geq d\binom{|U|}{k}-\eta n^k.$
Given a $k$-graph $F$, the \emph{uniform Tur\'{a}n density} of $F$ is defined as 
\begin{align*}\label{dot-turan-dense}
\pi_{\points}(F) = \sup \{ d\in [0,1] & : \text{for\ every\ } \eta>0 \ \text{and\ } n_0\in \mathbb{N},\ \text{there\ exists\ an\ } F \text{-free}, \\
&\quad \text{uniformly}\ (d,\eta) \text{-dense}\ k \text{-graph\ } H\ \text{with\ } |V(H)|\geq n_0 \}.
\end{align*}
Erd\H{o}s and S\'{o}s conjectured that $\pi_{\points}(K_4^{(3)-})=\frac{1}{4}$, based on the following construction of Erd\H{o}s and Hajnal \cite{erdosconstruction}. Take a tournament $T$ on a large vertex set $V$ uniformly at random, define a $3$-graph $H(T)$ on the same vertex set in which each edge corresponds to a cyclic triangle in $T$. One can easily check that $H(T)$ contains no copy of $K_4^{(3)-}$ and for every $\eta>0$ the $3$-graph $H(T)$ is uniformly $(\frac{1}{4},\eta)$-dense with high probability. This conjecture was recently confirmed by Glebov, Kr\'{a}l' and Volec \cite{K43-dan} using flag algebra, and Reiher, R\"{o}dl and Schacht \cite{K43-rodl} later gave an alternative proof via hypergraph regularity method. For $K_4^{(3)}$, a construction due to R\"{o}dl \cite{K43rodl} shows that $\pi_{\points}(K_4^{(3)})\geq\frac{1}{2}$, but whether $\frac{1}{2}$ is the correct value of $\pi_{\points}(K_4^{(3)})$ still remains open. Reiher, R\"{o}dl and Schacht \cite{reiher2018hypergraphs} characterized all $3$-graphs with vanishing uniform Tur\'{a}n density, and provided a construction showing that the uniform Tur\'{a}n density cannot lie in the interval $(0,\frac{1}{27})$. Later, Garbe, Kr\'{a}l' and Lamaison \cite{1/27dan} found a specific family of $3$-graphs $F$ with $\pi_{\points}(F)=\frac{1}{27}$. Recently, the uniform Tur\'{a}n density of the $3$-uniform tight cycle $C_{\ell}^{(3)}$ on $\ell$ vertices with $\ell\geq5$ was determined by Buci\'{c}, Cooper, Kr\'{a}l', Mohr and Correia \cite{cycledan}. For more results and problems on this topic and other variants, we refer the readers to~\cite{zhou1,zhou2,reihersurvey,K43dotedge,reiher2018hypergraphs,manteluniform,K43-rodl}.

In this paper, we consider the following codegree variation. Given a $k$-graph $H$ and a subset $S$ of $V(H)$ with $1\leq|S|\leq k-1$, define the \emph{degree} of $S$ in $H$, denoted by $d_{H}(S)$, as the number of edges containing $S$, i.e., $d_{H}(S)=|\{e\in E(H)~:~S\subseteq e\}|$. The \emph{minimum codegree} $\delta_{\co}(H)$ of $H$ is the minimum of $d_{H}(S)$ over all $(k-1)$-subsets $S$ of $V(H)$. Given a $k$-graph $F$ and $n\in \mathbb{N}$, the \textit{codegree Tur\'{a}n number} $\ex_{\co}(n,F)$ is the maximum $\delta_{\co}(H)$ an $n$-vertex $F$-free $k$-graph $H$ can admit. Similarly, the \textit{codegree Tur\'{a}n density} $\pi_{\co}(F)$ of $F$ is defined as 
$$\pi_{\co}(F)=\lim\limits_{n\to \infty}\frac{\ex_{\co}(n,F)}{n}.$$ 
This limit always exists, and it is not hard to see that $\pi_{\co}(F)\leq\pi(F)$ for any $k$-graph $F$. In particular, $\pi_{\co}(F)=\pi(F)$ when $F$ is a graph.

Determining the codegree Tur\'{a}n density of a $k$-graph with $k\geq3$ seems also very difficult in general. In the late 1990s, Nagle \cite{K43-nagle} and Czygrinow and Nagle \cite{K43nagle} conjectured that $\pi_{\co}(K_4^{(3)-})=\frac{1}{4}$ and $\pi_{\co}(K_4^{(3)})=\frac{1}{2}$, respectively. Recently, the $\pi_{\co}(K_4^{(3)-})$ case was settled by Falgas-Ravry, Pikhurko, Vaughan and Volec \cite{K43-codegree} in a stronger form via flag algebra technique. There are few other 3-graphs whose codegree Tu\'ran degree are known: Fano plan by Mubayi \cite{mubayi2005co}, $F_{3,2}$ with $V(F_{3,2})=[5]$ and $E(F_{3,2})=\{123,124,125,345\}$ by Falgas-Ravry, Marchant, Pikhurko and Vaughan \cite{falgas2015codegree}, and $C_{\ell}^{(3)-}$, the tight cycle with one edge removed,
by Piga, Sales and Sch\"{u}lke \cite{piga2022codegree}.

For the usual Tur\'{a}n density and the uniform Tur\'{a}n density for $3$-graphs, zero is a jump, i.e.~not an accumulation point. In constrast, recent work of Piga and Sch\"{u}lke \cite{piga2023hypergraphs} showed that surprisingly the codegree Tur\'{a}n density can be arbitrarily close to zero.

We study the relationship between the codegree Tur\'{a}n density of a $3$-graph and its uniform Tur\'{a}n density. Recall that $\pi_{\mathrm{co}}(K_4^{(3)-})=\pi_{\points}(K_4^{(3)-})=\frac{1}{4}$ and it is conjectured that $\pi_{\mathrm{co}}(K_4^{(3)})=\pi_{\points}(K_4^{(3)})=\frac{1}{2}$ with the same extremal construction. Falgas-Ravry, Pikhurko, Vaughan and Volec \cite{K43-codegree} asked the following question.

\begin{question*}[\cite{K43-codegree}]\label{ques}
For any $3$-graph $F$, is it true that $\pi_{\points}(F)\leq\pi_{\mathrm{co}}(F) ?$
\end{question*}

Falgas-Ravry and Lo \cite{loanswer} provided a positive answer to this question under a stronger uniform denseness assumption. 

In this paper, we give a positive answer to the above question when $\pi_{\mathrm{co}}(F)=0$.  

\begin{theorem}\label{main}
    Let $F$ be a $3$-graph. If $\pi_{\mathrm{co}}(F)=0$, then $\pi_{\points}(F)=0$.
\end{theorem}

Our result implies the following relationships among the three densities we discussed:
$$\{F:3\text{-graph with}~\pi(F)=0\}\subseteq\{F:3\text{-graph with}~\pi_{\mathrm{co}}(F)=0\}\subseteq\{F:3\text{-graph with}~\pi_{\points}(F)=0\}.$$

\begin{figure}[!ht]
\centering
\begin{tikzpicture}[scale=0.8]
\draw[fill=pink,fill opacity=0.025](-4,-0.5) rectangle (4,2.5) ;
\draw[fill=yellow,fill opacity=0.05](0,0.6) ellipse(3 and 1);
\draw[fill=lime,fill opacity=0.1](0,0.25) ellipse(1.7 and 0.5);
\node[inner sep= 1.3pt,black](u) at (0,0.32) {$\pi(F)=0$};
\node[inner sep= 1.3pt,black](u) at (0,1.2) {$\pi_{\mathrm{co}}(F)=0$};
\node[inner sep= 1.3pt,black](u) at (0,2) {$\pi_{\points}(F)=0$};
\end{tikzpicture}
\end{figure}

\noindent{\bf Our approach.~} A natural approach is to utilize a vanishing unform Tur\'an condition $(\spadesuit)$ by Reiher, R{\"o}dl and Schacht \cite{reiher2018hypergraphs} (see~\cref{rodl}), and look for hypergraphs satisfying $(\spadesuit)$ and linear minimum codegree. This, however, fails to work as there are uniformly dense hypergraphs not containing any subhypergraph with linear minimum codegree, see the discussion in~\cref{sec:fail}. Instead of working with $(\spadesuit)$, we observe a reformulation of $(\spadesuit)$ using monotone paths in the link graphs (\cref{permutation}) and a variant for `bipartite' 3-graphs (\cref{strengthen}). A key step in our proof is to show that we may further assume that the forbidden `bipartite' 3-graph $F$ possesses certain connectedness condition (\cref{connected}), which allows us to force a clustering phenomenon (\cref{cl:cluster}) in our random geometric construction to avoid $F$. We then link this graph cyclically to obtain the final construction.


\medskip

\noindent{\bf Notations.~}Let $F$ be a 3-graph. For a vertex $v\in V(F)$, its \emph{link graph} in $F$, denoted by $L_F(v)$, is the graph with $V(L_F(v))=V(F)$ and $E(L_F(v))=\{uw~:~uvw\in E(F)\}$. The \emph{shadow graph} of $F$, denoted by $\partial F$, is the graph with $V(\partial F)=V(F)$ and $E(\partial F)=\{uw~:~uvw\in E(F)\}$. Let $u,v$ be two vertices of the $3$-graph $F$. We say that $w\in V(F)$ is a \emph{coneighbor} of $u,v$ if $uvw\in E(F)$. The \emph{coneighbor set} of $u,v$ is defined as $N_{F}(uv)=\{w~:~uvw\in E(F)\}$. 

Let $S$ be a finite set, we say that $\sigma$ is a \emph{labeling} of $S$ if $\sigma:S\rightarrow [|S|]$ is a bijection. Let $G$ be a graph and $\sigma$ be a labeling of $V(G)$. Let $u,v,w\in V(G)$ and $uvw$ form a path of length two in $G$. We say that $uvw$ is a \emph{monotone $P_3$} if $\sigma(u)<\sigma(v)<\sigma(w)$ or $\sigma(u)>\sigma(v)>\sigma(w)$. The following operation provides us a natural way to merge several labelings into a larger one. Let $S_1,S_2$ be two disjoint finite sets and let $\sigma_1,\sigma_2$ be two labelings of $S_1,S_2$ respectively. The sum of $\sigma_1$ and $\sigma_2$, denoted by $\sigma_1 \oplus \sigma_2$, is a labeling of $S_1\cup S_2$ where
\[\sigma_1 \oplus \sigma_2(s)= \left \{
\begin{array}{ll}
 \sigma_1(s),  & s\in S_1;\\
 \sigma_2(s)+|S_1|, & s\in S_2.
\end{array}
\right.\]
For more than two labelings $\sigma_1,\sigma_2,\ldots,\sigma_k$, the sum of them, denoted by $\sum_{i=1}^{k}\sigma_i$, is inductively defined by 
\[\sum_{i=1}^{k}\sigma_i=\left(\sum_{i=1}^{k-1}\sigma_i\right)\oplus \sigma_k.\]
Let $S$ be a finite set and $\sigma':S\rightarrow \mathbb{Z}$ be an injection. We say that a labeling $\sigma$ of $S$ is \emph{induced by} $\sigma'$ if for every $s_1,s_2\in S$, $\sigma(s_1)>\sigma(s_2)$ if and only if $\sigma'(s_1)>\sigma'(s_2)$. Obviously, $\sigma$ exists and is unique.

\section{Auxiliary results}

\subsection{A first attempt}\label{sec:fail}

 To prove Theorem \ref{main}, a natural idea is to utilize the following characterization of $3$-graphs with vanishing uniform Tur\'{a}n density due to Reiher, R{\"o}dl and Schacht \cite{reiher2018hypergraphs}. 

\begin{theorem}[Reiher, R{\"o}dl and Schacht, \cite{reiher2018hypergraphs}]\label{rodl}
For any $3$-graph $F$, the following are equivalent.
\begin{itemize}
    \item[$(\clubsuit)$] $\pi_{\points}(F)=0$;
    \item[$(\spadesuit)$] there is an enumeration of the vertex set $V(F)=\{v_1,...,v_f\} $ and there is a $3$-colouring $\phi: \partial F \rightarrow \{\emph {red,blue,green}\}$ of the pairs of vertices covered by edges of $F$ such that every edge $\{v_i,v_j,v_k\}\in E(F)$ with $i<j<k$ satisfies
$$\phi(v_i v_j ) = \emph{red},~\phi(v_i v_k) = \emph{blue},~\phi(v_j v_k) = \emph{green}.$$
\end{itemize}
\end{theorem}

Fix an arbitrary $3$-graph $F$ with $\pi_{\co}(F)=0$. Suppose for some $\varepsilon>0$ and infinitely many integers $n$, one can construct $n$-vertex $3$-graphs $H_n$ satisfying~$(\spadesuit)$ and $\delta_{\co}(H_n)\geq\varepsilon n$. Then as $\pi_{\co}(F)=0$, $F$ must be a subhypergraph of $H_n$ for some sufficiently large $n$. Inheriting from $H_n$, $F$ also satisfies condition $(\spadesuit)$, implying that $\pi_{\points}(F)=0$ by Theorem \ref{rodl} as desired. Unfortunately, we note that this idea is not feasible by the following simple observation.  

 \begin{observation}
Let $H$ be any $3$-graph satisfying condition $(\spadesuit)$ in Theorem \ref{rodl}. Then $\delta_{\co}(H')\leq 1$ for any $H'\subseteq H$.
 \end{observation}
 \begin{proof}
Clearly, $H'$ also satisfies condition $(\spadesuit)$. Let $v_1,v_2,\ldots,v_{n'}$ be an enumeration of $V(H')$ and $\phi$ be a $3$-coloring of $\partial H'$ satisfying condition $(\spadesuit)$. If $d_{H'}(\{v_1,v_2\})=0$, then we are done. Otherwise, let $v_1v_2v_{i}$ be an edge of $H'$. Then condition $(\spadesuit)$ implies that $\phi(v_2v_{i})= \rm{green}$, and therefore $d_{H'}(\{v_2,v_i\})=1$.
\end{proof}

By a construction in \cite{reiher2018hypergraphs}, for every $\eta>0$, there is an infinite sequence of $(\frac{1}{27},\eta)$-dense $3$-graphs satisfying condition $(\spadesuit)$, which together with the above observation provides a negative answer to the following question asked by Falgas-Ravry, Pikhurko, Vaughan and Volec \cite{K43-codegree}.

\begin{question}
Let $(H_n)_{n\in\mathbb{N}}$ be a sequence of uniformly dense $3$-graphs with density $d>0$. Must there exist a sequence of subhypergraphs $(H'_n)_{n\in\mathbb{N}}$ with $H'_n\subseteq H_n$, $|V(H'_n)|\rightarrow\infty$ and $\delta_{\co}(H'_n)/|V(H'_n)|$ bounded away from zero?
\end{question}

\subsection{A bipartite vanishing  condition}
In our approach, we do not insist on satisfying $(\spadesuit)$ globally. We will construct infinitely many $3$-graphs satisfying $(\spadesuit)$ locally. 
For this purpose, the following reformulation of~\cref{rodl} using the structure property of link graphs is more helpful for us.
\begin{lemma} \label{permutation}
For any $3$-graph $F$, the following are equivalent.
\begin{itemize}
    \item $\pi_{\points}(F)>0$; 
    \item for every labeling $\sigma$ of $V(F)$, there is some vertex $v\in V(F)$ such that $L_{F}(v)$  contains a monotone $P_3$.
\end{itemize}
\end{lemma}

\begin{proof} 
    Let $f=|V(F)|$. For one side, suppose $\pi_{\points}(F)>0$, and to the contrary that there exists a labeling $\sigma$ of $V(F)$ such that for every $v\in V(F)$, $L_{F}(v)$ does not contain a monotone $P_3$. Then this labeling induces an enumeration of $V(F)=\{v_1,v_2,\ldots,v_f\}$ by letting $v_i=\sigma^{-1}(i)$. For any $v_iv_j\in E(\partial F)$ with $i<j$ and $v_k\in N_F(v_iv_j)$, let
    \begin{equation} \phi(v_iv_j)=
    \begin{cases}
    \rm{red},   & k>j;\\
    \rm{blue},  & i<k<j;\\
    \rm{green}, & k<i.\nonumber
    \end{cases}
    \end{equation}
    Now we show that $\phi$ is well defined. Otherwise, there are $v_{k_1},v_{k_2}\in N_F(v_iv_j)$ such that $k_1<j<k_2$ or $k_1<i<k_2$ holds, then either $v_{k_1}v_jv_{k_2}$ is a monotone $P_3$ in $L_F(v_i)$ or $v_{k_1}v_iv_{k_2}$ is a monotone $P_3$ in $L_F(v_j)$, which contradicts to our assumption. But this coloring $\phi$ and the enumeration given by $\sigma$ implies $\pi_{\points}(F)=0$ by \cref{rodl}, which is also a contradiction.

    For the other side, suppose that for every labeling $\sigma$ of $V(F)$, there is some vertex $v\in V(F)$ such that $L_{F}(v)$ contains a monotone $P_3$. Let  $V(F)=\{v_1,v_2,\ldots,v_f\}$  be any enumeration of $V(F)$ and $\sigma$ be a labeling of $V(F)$ by setting $\sigma(v_i)=i$ for $i\in[f]$. Then by our assumption, there is a monotone $P_3=v_iv_jv_k$ with $i<j<k$ in $L_F(v_\ell)$ for some $\ell\in [f]$. Suppose that a coloring $\phi$ in \cref{rodl} exists. Now we consider the color of $v_{\ell}v_j$. If $\ell<j$, then $\{v_{\ell},v_j,v_k\}\in E(F)$ implies that $\phi(v_{\ell}v_j)=\rm{red}$ while $\{v_{\ell},v_i,v_j\}\in E(F)$ implies that $\phi(v_{\ell}v_j)=\rm{blue~or~green}$, a contradiction. If $\ell>j$, then $\{v_j,v_{\ell},v_k\}\in E(F)$ implies that $\phi(v_{\ell}v_j)=\rm{red~or~blue}$ while $\{v_i,v_j,v_{\ell}\}\in E(F)$ implies that $\phi(v_{\ell}v_j)=\rm{green}$, again a contradiction. Therefore, $\pi_{\points}(F)>0$ follows.
\end{proof}  

A 3-graph $F$ is said to be a \emph{$(2,1)$-type $3$-graph} with respect to a partition $V(F)=A\cup B$ if any edge $e\in E(F)$ has two vertices in $A$ and one vertex in $B$. For convenience, when we mention a $(2,1)$-type $3$-graph with respect to a partition $X\cup Y$ in the following, we always mean that each edge of this $3$-graph has two vertices in $X$ and one vertex in $Y$. That is, the order of $X,Y$ in the union matters. To prove our main theorem, we first strengthen \cref{permutation} for $(2,1)$-type 3-graphs.

\begin{lemma}\label{strengthen}
        Let $F$ be a $(2,1)$-type $3$-graph with respect to $V(F)=A\cup B$. Then $\pi_{\points}(F)>0$ if and only if for every labeling of  $A$, there exists $u\in B$ such that $L_F(u)$ contains a monotone $P_3$. 
\end{lemma}

\begin{proof}
For one side, suppose $\pi_{\points}(F)>0$. Let $\sigma_1$ be a labeling of $A$ and $\sigma_2$ be a labeling of $B$,  $\sigma=\sigma_1\oplus\sigma_2$.
By \cref{permutation}, there is some $u\in A\cup B$ such that $L_F(u)$ contains a monotone $P_3$ with respect to $\sigma$. If $u\in A$, then $L_{F}(u)$ is a bipartite graph on two parts $A$ and $B$. In this case, any $P_3$ in $L_{F}(u)$ has two end points lying in one part and the mid point lying in the other part. Since any vertex in $B$ has a larger label than any vertex in $A$, all of $P_3$ in $L_{F}(u)$ are not monotone with respect to $\sigma$. Therefore, $u\in B$ and the monotone $P_3$ (with respect to $\sigma$) lies in $A$ and is also monotone with respect to $\sigma_1$.

For the other side, let $\sigma$ be a labeling of $A\cup B$. Restricting $\sigma$ to $A$ induces a labeling of $A$, denoted by $\sigma_0$. By the assumption, there is some $u\in B$ such that $L_F(u)$ contains a monotone $P_3$ with respect to the labeling $\sigma_0$. Note that this monotone $P_3$ is also monotone with respect to the labeling $\sigma$. Hence, $\pi_{\points}(F)>0$ follows from \cref{permutation}, which finishes the proof. 
\end{proof}

\begin{definition}[Auxiliary graph $\Gamma$]\label{defn:aux}
Let $F$ be a $(2,1)$-type $3$-graph with respect to $A\cup B$. Define an auxiliary graph $\Gamma_F$ as follows: $V(\Gamma_F)=A$, and for any two vertices $u,v\in A$, $uv$ is an edge of $\Gamma_F$ if and only if there is $w\in A$ and $x\in B$ such that $uwx, vwx\in E(F)$.    
\end{definition}

 In other words, two vertices in $A$ are joined in $\Gamma_F$ if they are the endpoints of a $P_3$ in the link graph of some vertex in $B$. Assisted by this auxiliary graph, we can extract a subhypergraph $F'$, which admits certain `connectedness' property and satisfies $\pi_{\points}(F')>0$, from any $(2,1)$-type $3$-graph $F$ with $\pi_{\points}(F)>0$. This `connectedness' property of $F'$ allows us to prove a cluserting phenomenon in our construction.

\begin{lemma} \label{connected}
    Let $F$ be a $(2,1)$-type $3$-graph with respect to $V(F)=A\cup B$. If $\pi_{\points}(F)>0$, then there is an induced subhypergraph $F'\subseteq F$ with $\pi_{\points}(F')>0$ such that $\Gamma_{F'}$ is a subgraph of some connected component of $\Gamma_F$.
\end{lemma}
\begin{proof}
    Let $A_1,A_2,\ldots,A_k\subseteq A$ be the vertex sets of all the connected components of $\Gamma_F$. For each $i\in[k]$, let 
    
    \[B_i=\{b~:~ba_1a_2\in E(F),b\in B,a_1,a_2\in A_i\},\]
    
    \noindent and $F_i=F[A_i\cup B_i]$. Then one can easily check that $\Gamma_{F_i}$ is a subgraph of the connected component of $\Gamma_F$ on vertex set $A_i$. Now it suffices to show that there is some $i\in[k]$ such that $\pi_{\points}(F_i)>0$. 

    We prove it by contradiction. If $\pi_{\points}(F_i)=0$ for all $i\in[k]$, by \cref{strengthen}, there is a labeling $\sigma_i$ of $A_i$ such that for every $b\in B_i$, there is no monotone $P_3$ in $L_{F_i}(b)$ with respect to $\sigma_i$.  Setting $\sigma=\sum_{i=1}^{k}\sigma_i$, we next show that for every $b\in B$, there is no monotone $P_3$ in $L_F(b)$ with respect to $\sigma$, which contradicts to \cref{strengthen}.

    Let $xay$ be a $P_3$ in $L_F(b)$ for some $b\in B$ with $x\in A_i,a\in A_j,y\in A_{\ell}$. By the definition of $\Gamma_F$, $xy$ is an edge of $\Gamma_F$, and therefore $i=\ell$. If $i=j$, then $x,a,y\in A_i$, $b\in B_i$ and $xay$ can not be a monotone $P_3$ with respect to $\sigma$ since it is not monotone with respect to $\sigma_i$.  If $i\neq j$, then either 
    $\sigma(a)>\max\{\sigma(x),\sigma(y)\}$ or $\sigma(a)<\min\{\sigma(x),\sigma(y)\}$, and therefore $xay$ is again not a monotone $P_3$. Combining these cases, we conclude that there is no monotone $P_3$ with respect to $\sigma$ in $L_F(b)$.   
\end{proof}


\section{Proof of Theorem \ref{main}}

Let $F$ be a $3$-graph. A vertex partition $A\cup B\cup C$ of $V(F)$ is said to be \textit{good} if $F[A\cup B],F[B\cup C],F[C\cup A]$ are all $(2,1)$-type $3$-graphs, and there are no other edges in $F$. The following lemma shows that in a $3$-graph $F$ with $\pi_{\points}(F)>0$, if $V(F)$ admits a good vertex partition,  then we can always find a $(2,1)$-type subhypergraph $F'$ with $\pi_{\points}(F')>0$. 

\begin{lemma}\label{3parts}
Let $F$ be a $3$-graph with $\pi_{\points}(F)>0$. If there exists a good vertex partition $A\cup B\cup C$ of $V(F)$, then $\pi_{\points}\left(F[A\cup B]\right)>0$ or $\pi_{\points}\left(F[B\cup C]\right)>0$ or $\pi_{\points}\left(F[C\cup A]\right)>0.$    
\end{lemma}
\begin{proof}
Let $F_1,F_2$ and $F_3$ denote the three induced $(2,1)$-type $3$-graphs $F[A\cup B],F[B\cup C]$ and $F[C\cup A]$, respectively. Suppose to the contrary that $\pi_{\points}(F_i)=0$ for all $1\leq i\leq 3$. Then by \cref{strengthen}, there are three labelings $\sigma_{A},\sigma_{B},\sigma_{C}$ of the three parts $A$,$B$,$C$, respectively, such that for every $a\in A$, $b\in B$ and $c\in C$, no monotone $P_3$ exists in $L_{F_1}(b)$, $L_{F_2}(c)$ and $L_{F_3}(a)$. Let $\sigma=\sigma_{A}\oplus \sigma_{B}\oplus \sigma_{C}$. Now we prove that for every $v\in V(F)$, $L_F(v)$ does not contain a monotone $P_3$ with respect to $\sigma$, and hence a contradiction follows from \cref{permutation}.

For any $b\in B$, note that $L_F(b)$ is the disjoint union of $L_{F_1}(b)$ and $L_{F_2}(b)$. Since there is no monotone $P_3$ in $L_{F_1}(b)$ with respect to the labeling $\sigma_A$, there is no monotone $P_3$ in $L_{F_1}(b)$ with respect to the labeling $\sigma$. On the other hand, noting that $L_{F_2}(b)$ is a bipartite graph with partition $B\cup C$, so the two endpoints of any $P_3$ in $L_{F_2}(b)$ are both located in one part while the mid point is located in the other part. Since $\sigma=\sigma_{A}\oplus\sigma_{B}\oplus\sigma_{C}$, one can easily see that no monotone $P_3$ exists in $L_{F_2}(b)$. Consequently, for any $b\in B$, there is no monotone $P_3$ in $L_F(b)$ with respect to the labeling $\sigma$. By similar arguments, the same holds for any $x\in A\cup C$ as desired.
\end{proof}


We start with a `bipartite' construction.

\begin{theorem}\label{good}
For any positive integer $k$,  there exists $\varepsilon=\varepsilon_k>0$ such that for infinitely many integers $n$, there exists a $(2,1)$-type $3$-graph $H$ with respect to a vertex partition $V(H)=A\cup B$ such that $|A|=|B|=n$ and the following conditions are satisfied.
\begin{enumerate}
    \item[\emph{(C1)}] $d_H(\{a_1,a_2\})\geq\varepsilon n$ for any two vertices $a_1,a_2\in A$;
    \item[\emph{(C2)}] $d_H(\{a,b\})\geq\varepsilon n$ for every $a\in A$ and $b\in B$;
    \item[\emph{(C3)}] $H$ contains no copy of any $3$-graph $F$ with $|V(F)|\le k$ and $\pi_{\points}(F)>0$.
\end{enumerate}
\end{theorem}

\begin{proof}
Set $\varepsilon_k=\frac{1}{32k^2}$, and let $n$, $q$, $r$ be integers such that $r=4k$, $n=qr+1$ and $n\gg r$. All the indices in the following construction are modulo $n$.

We first construct a random graph $G$ on vertex set $\{a_0,a_1,\ldots,a_{n-1}\}$. For every $0\leq i\leq n-1$ and $0\leq j\leq r-1$, let 
\[S_{ij}=\{a_{i+t}~:~t\in\{jq+1,jq+2,\cdots,(j+1)q\}\}.\] 
Let $X_0,X_1,\ldots,X_{n-1}$ be $n$ independent random variables, each of which takes a value from $\{0,1,\ldots, r-1\}$ uniformly at random. Define our random graph $G$ as follows: $V(G)=\{a_0,a_1,\ldots,a_{n-1}\}$, and for every $0\leq i <j \leq n-1$, $a_{i}a_{j}\in E(G)$ if and only if $a_i\in S_{jX_{j}}$ and $a_j\in S_{iX_i}$. Note that $a_ia_j$ forms an edge with probability $\frac{1}{r^2}$.

Let $G_0,G_1,\ldots,G_{n-1}$ be i.i.d. copies of $G$ on the same vertex set $\{a_0,a_1,\ldots,a_{n-1}\}$. Define a $(2,1)$-type $3$-graph $H$ with vertex partition $A\cup B$ as follows: $A=\{a_0,a_1,\ldots,a_{n-1}\}$, $B=\{b_0,b_1,\ldots,b_{n-1}\}$, and $a_ia_jb_{\ell}\in E(H)$ if and only if $a_ia_j\in E(G_{\ell})$. In other words, $G_{\ell}$ is the link graph of $b_{\ell}$ for any $0\leq\ell\leq n-1$.

Finishing our proof, we next show that (C1) and (C2) are satisfied with high probability and  (C3) is guaranteed by a geometric property of $H$.  
\bigskip

\noindent{\bf Verifying (C1)} 
Note that for any $a_i,a_j\in A$, $d_H(\{a_i,a_j\})$ is exactly the number of $G_{\ell}$ containing $a_ia_j$ as an edge. Since $\mathbb{P}\left(a_ia_j\in G\right)=\frac{1}{r^2}$ and $G_{\ell}$ is an i.i.d. copy of $G$, it follows that 
$d_H(\{a_i,a_j\})\sim {\rm Bin}(n,\frac{1}{r^2}).$
It follows from Chernoff's inequality that $\mathbb{P}\left(d_H(\{a_i,a_j\})\leq \frac{n}{2r^2}\right)\leq e^{-\frac{n}{8r^2}}.$ Therefore, taking a union bound over $\binom{n}{2}$ such pairs, the probability that $d_H(\{a_i,a_j\})> \frac{n}{2r^2}$ for every pair $\{a_i,a_j\}$ is at least $1-\binom{n}{2}e^{-\frac{n}{8r^2}}$.
\bigskip

\noindent{\bf Verifying (C2).} 
Let $a_i\in A$, $b_{\ell}\in B$. Since $d_H\left(\{a_i,b_{\ell}\}\right)=d_{G_{\ell}}(a_i)$ and $G_{\ell}$ is an i.i.d. copy of $G$, $d_H\left(\{a_i,b_{\ell}\}\right)$ and $d_G(a_i)$ are identically distributed. Conditioning on the choice of $X_i$, for each $a_j\in S_{iX_i}$, the probability that $a_ia_j\in E(G)$ is exactly $\frac{1}{r}$ and for each $a_j\notin S_{iX_i}$, the probability that $a_ia_j\in E(G)$ is zero. Since the choice of each $X_j$ is independent, $d_G(a_i)\sim {\rm Bin}(q,\frac{1}{r})$.  
Thus, $\mathbb{E}\left(d_H\left(\{a_i,b_{\ell}\}\right)\right)=\frac{n}{r^2}$. It follows from Chernoff's inequality that $\mathbb{P}\left(d_H\left(\{a_i,b_{\ell}\}\right)\leq \frac{n}{2r^2}\right)\leq e^{-\frac{n}{8r^2}}.$  Therefore, taking a union bound over $n^2$ such pairs, the probability that $d_H\left(\{a_i,b_{\ell}\}\right)> \frac{n}{2r^2}$ for every $a_i$ and $b_{\ell}$ is at least $1-n^2e^{-\frac{n}{8r^2}}$.
\bigskip

\noindent{\bf Verifying (C3).} 
Let $F$ be a $3$-graph with $|V(F)|\le k$ and $\pi_{\points}(F)>0$, and suppose to the contrary that $H$ contains a copy of $F$. Inheriting from the structure of $H$, $F$ must be a $(2,1)$-type $3$-graph. By Lemma \ref{connected}, for some $A'\subset A$ and $B'\subset B$, the induced subhypergraph $F':=F[A'\cup B']$ of $F$ satisfies that $\pi_{\points}(F')>0$ and $\Gamma_{F'}$ is a subgraph of some connected component of $\Gamma_F$. Let us imagine that the elements of $A$ are distributed on the unit circle of the complex plane with $a_{\ell}$ being the point $e^{2\pi \ell i/n}$. The following claim shows a `clustering' property of $F'$.

\begin{claim}[Clustering]\label{cl:cluster}
 For every $b\in B$ and $a_j\in A$, all the neighbors of $a_j$ in $L_{H}(b)$ lie in an arc of length at most $\frac{2\pi}{r}$ not containing $a_j$. Consequently, all the vertices of $A'$ lie on an open half circle.  
\end{claim}
\begin{proof}
 The first part of the claim follows from the constructions of the random graph $G$ and the $3$-graph $H$. Indeed, all the neighbors of $a_j$ in $L_{H}(b)$ concentrate on the random arc $S_{jX_j}$ of length $\frac{q-1}{qr+1}\cdot 2\pi\le \frac{2\pi}{r}$, and note that the arc $S_{jX_j}$ does not contain $a_j$. 

For the second part, by the definition of $\Gamma_F$ and the assumption that $F\subseteq H$, each edge $xy$ in any component of $\Gamma_F$ implies that there is a path $xay$ in $L_{H}(b)$ for some $a\in A$ and $b\in B$. So by the first part, the distance of $x,y$ on the unit circle is at most $\frac{2\pi}{r}$. Since $\Gamma_{F'}$ is a subgraph of some connected component of $\Gamma_F$ and $|A'|\leq |V(F)|\leq \frac{r}{4}$, we conclude that every pair of vertices in $A'$ has a distance less than $\pi$ on the unit circle, which implies that all the elements of $A'$ lie on an open half circle.  
\end{proof}

Let $\mathcal{C}$ be an open half circle containing all the vertices of $A'$. Along the counter-clockwise direction of $\mathcal{C}$, we obtain an enumeration $x_1,x_2,\ldots,x_{|A'|}$ of all the vertices in $A'$. Let $\sigma$ be the labeling of $A'$ defined by setting $\sigma(x_j)=j$ for all $1\leq j\leq |A'|$. Recall that $\pi_{\points}(F')>0$. By Lemma \ref{strengthen}, for some $b\in B'$, there is a monotone path $x_{\ell}x_{s}x_{t}$ in $L_{F'}(b)$ with $\ell<s<t$. By the first part of~\cref{cl:cluster}, $x_{\ell}$ and $x_t$ lie in some arc $\mathcal{R}$ of length at most $2\pi/r$ not containing $x_s$. Since $x_s\not\in\mathcal{R}$, $\mathcal{R}$ has to contain the complement of $\mathcal{C}$, which has length more than $\pi$, a contradiction.
\end{proof}

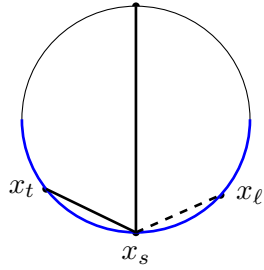
\begin{figure}[!ht]
\centering
\begin{tikzpicture}[scale=0.3]
\node[ inner sep= 1.0pt,black](u) at (0,5) [circle,fill]{};
\node[inner sep= 1.0pt,black](u) at (0,-5)[circle,fill]{};
\node[inner sep= 1.3pt,black](u') at (0,-6)[]{$x_s$};
\node[ inner sep= 1.0pt,black](u) at (-3.940,-3.078) [circle,fill]{};
\node[inner sep= 1.3pt,black](u') at (-5.0,-3.078)[]{$x_{t}$};
\node[inner sep= 1.0pt,black](u) at (3.716,-3.346)[circle,fill]{};
\node[inner sep= 1.3pt,black](u') at (5,-3.346)[]{$x_{\ell}$};



  \draw[fill=white,fill opacity=0.05](0,0) circle(5);
  \draw[blue,line width=1] (5,0) arc (0:-180:5);
  \draw [black,line width=1](0 ,-5)-- (0,5);  
  \draw [black,dashed,line width=1](0 ,-5)-- (3.716,-3.346);
 \draw [black,line width=1](0,-5)-- (-3.940,-3.078);
 
  
 \end{tikzpicture}
\caption{An illustration of the geometric property of circles that if $x_t$, $x_s$ and $x_{\ell}$ lie on a (blue) half circle, then $x_t$ and $x_{\ell}$ lie on the distinct half circles determined by $x_s$ and its opposite point. Therefore, the construction of $G$ implies that $x_{\ell}$ and  $x_t$ can not be jointed to $x_s$ simultaneously. }
\end{figure}

A remark here is that requiring that $n=qr+1$ is just for our convenience to present the proof, our arguments actually hold for all sufficiently large integers $n$ with a tiny modification. Combining all, we now give the proof of our main theorem.

\begin{proof}[{\bf Proof of Theorem \ref{main}}] 
Let $F$ be a $k$-vertex 3-graph with $\pi_{\points}(F)>0$. To prove Theorem \ref{main}, it suffices to show $\pi_{\mathrm{co}}(F)>0$. Let $\varepsilon_k$ be the constant from \cref{good}, and let $\hat{H}$ be a $2n$-vertex $(2,1)$-type 3-graph satisfying all (C1), (C2) and (C3). 

Let $H$ be a $3n$-vertex 3-graph which admits a good vertex partition $V(H)=A \cup B\cup C$ with $|A|=|B|=|C|=n$ such that $H[A\cup B], H[B\cup C]$ and $H[C\cup A]$ are all isomorphic to $\hat{H}$.  
By (C1) and (C2), $\delta_{\co}(H)\ge \varepsilon_k n$. It remains to show that $F$ is not a subhypergraph of $H$.
Suppose to the contrary that $H$ contains $F$. Inheriting from the structure of $H$, there is a good partition $A'\cup B'\cup C'$ of $V(F)$ with $A'\subseteq A, B'\subseteq B$ and $C'\subseteq C$.
By \cref{3parts}, without loss of generality, we may assume that $\pi_{\points}(F[A'\cup B'])>0$. Since $|A'\cup B'|\leq |V(F)|=k$, by \cref{good}, $F[A'\cup B']$ can not be a subhypergraph of $H[A\cup B]$, a contradiction. 
\end{proof}

\medskip
\noindent{\bf Acknowledgements.~}The authors would like to thank Suyun Jiang, Ruonan Li, Guanghui Wang and Zhuo Wu for helpful discussions.


\bibliographystyle{abbrv}
\bibliography{references.bib}
\end{document}